\documentclass[12pt]{article}
\usepackage{mathtools}
\usepackage[demo]{graphicx}
\usepackage[hidelinks]{hyperref}
\usepackage{multirow}
\usepackage{float}
\usepackage{diagbox}
\usepackage{indentfirst}
\usepackage{tikz}
\usepackage{amssymb}
\usepackage{amsmath,mathrsfs,amsfonts}
\usepackage{graphics}
\usepackage{amsthm}
\usepackage{setspace}
\usepackage{setspace}
\usepackage[mathscr]{eucal}
\usepackage{caption}
\usepackage{subcaption}
\usepackage{rotating}
\usetikzlibrary{calc,decorations.pathreplacing,fit}
\definecolor{rblockfill}{RGB}{246,247,249}
\tikzset{
    rpath/.style={line width=.9pt},
    rlocaledge/.style={draw=black!65,line width=.65pt},
    rchosenedge/.style={line width=1.05pt},
    rextra/.style={densely dashed,line width=1.2pt},
    rpathvertex/.style={circle,fill=black,inner sep=1.8pt,outer sep=0pt},
    rinnervertex/.style={circle,draw=black,fill=white,line width=.7pt,inner sep=2pt,outer sep=0pt},
    rchosenvertex/.style={circle,draw=black,fill=white,line width=1.2pt,inner sep=2.4pt,outer sep=0pt},
    rblockbox/.style={draw=black!55,fill=rblockfill,rounded corners=4pt,line width=.65pt}
}
\newtheorem{theorem}{Theorem}
\newtheorem{theo}[theorem]{Theorem}
\newtheorem{lem}[theorem]{Lemma}
\newtheorem{corollary}[theorem]{Corollary}

\newtheorem{proposition}[theorem]{Proposition}

 \theoremstyle{definition}
\newtheorem{defi}[theorem]{Definition}

\theoremstyle{remark}

\def\p{\overrightarrow{p}}

\newcounter{casenum}[theorem]

\newcounter{subcasenum}[theorem]

\newcounter{claimnum}[theorem]
\newcommand{\cla}{%
  \par
  \refstepcounter{claimnum}%
  \textbf{Claim \arabic{claimnum}}.\enspace\ignorespaces
}

\allowdisplaybreaks

\begin{document}
\thispagestyle{plain}
\enlargethispage{2pt}

\begin{center}
{\Large Extremal spectral radius of nonregular graphs\\
with a fixed odd maximum degree}
\end{center}
\pagestyle{plain}
\begin{center}
	{
		{\small Zejun Huang, Chenxi Yang \footnote{Corresponding author. \\ Email: zejunhuang@szu.edu.cn (Huang), yangchenxi2022@email.szu.edu.cn (Yang)}}\\[3mm]
		{\small School of Mathematical Sciences, Shenzhen University, Shenzhen 518060, China}\\
	}
\end{center}
\begin{center}
\begin{minipage}{140mm}
\begin{center}
{\bf Abstract}
\end{center}

{\small
For integers $n\ge3$ and $2\le\Delta\le n-1$, let $\lambda_1(n,\Delta)$ be the maximum adjacency spectral
radius among all connected nonregular graphs of order $n$ and maximum degree $\Delta$. Liu conjectured that  for
each fixed integer $\Delta\ge3$,
\[
\lim_{n\to\infty}n^2\bigl(\Delta-\lambda_1(n,\Delta)\bigr)=
\begin{cases}
(\Delta-1)\pi^2/4,&\text{if $\Delta$ is odd};\\
(\Delta-2)\pi^2/2,&\text{if $\Delta$ is even}.
\end{cases}
\]
He proved the case   $\Delta=3$ and $\Delta=4$. We prove   the  conjecture   when $\Delta\ge5$ is odd.

}

\end{minipage}
\end{center}

\section{Introduction and main result}

All graphs considered in this paper are finite, undirected and simple unless stated otherwise. For a graph $G$ of
order $n$, let $V(G)$ and $E(G)$ denote its vertex and edge sets, let $d_G(v)$ be the degree of a vertex $v$, and
let $\Delta(G)$ be the maximum degree. We write $A(G)$ for the adjacency matrix of $G$ and $\rho(G)$ for its
spectral radius. For a connected graph, a positive vector $x$ satisfying $A(G)x=\rho(G)x$ is called a
\emph{Perron vector} of $G$.

For integers $n\ge3$ and $2\le\Delta\le n-1$, let $\mathcal C(n,\Delta)$ be the family of all connected
nonregular graphs of order $n$ with maximum degree $\Delta$, and define
\[
\lambda_1(n,\Delta):=\max_{G\in\mathcal C(n,\Delta)}\rho(G).
\]
For each connected graph $G$, it is well known that $\rho(G)\le\Delta(G)$, with equality if and only if $G$ is regular. Thus
$\Delta(G)-\rho(G)$ measures how close $G$ is spectrally to a regular graph, and determining
$\Delta-\lambda_1(n,\Delta)$ asks how small this gap can be under the prescribed order and maximum degree.

Stevanovi\'c \cite{D} initiated the study of this problem by giving a positive lower bound $$\Delta-\lambda_1(n,\Delta)>\frac{1}{2 n(n \Delta-1) \Delta^2}.$$ Zhang \cite{Zhang} and Cioab\u{a}, Gregory and Nikiforov
\cite{CGN} subsequently improved the dependence on the graph parameters. Cioab\u{a} \cite{SMC} proved the
diameter bound
\[
\Delta-\lambda_1(n,\Delta)>\frac1{nD},
\]
where $D$ is the diameter of $G$. Liu, Shen and Wang \cite{LSW} established
\[
\Delta-\lambda_1(n,\Delta)\ge\frac{\Delta+1}{n(3n+2\Delta-4)}.
\]
Further estimates were obtained by Liu, Huang and You \cite{LHY}, Zhang \cite{WZ}, and Feng and Zhang \cite{FZ}.
For fixed $\Delta$, these results naturally lead to the determination of the leading constant on the scale $n^{-2}$.

Liu, Shen and Wang \cite{LSW} conjectured that
\[
\lim_{n\to\infty}\frac{n^2\bigl(\Delta-\lambda_1(n,\Delta)\bigr)}{\Delta-1}=\pi^2
\]
for each fixed $\Delta$. Liu \cite{LLL} disproved this conjecture for all $\Delta\ge3$ by constructing graphs
that satisfy
\begin{equation}\label{eq:liu-construction}
\limsup_{n\to\infty}n^2\bigl(\Delta-\lambda_1(n,\Delta)\bigr)\le
\begin{cases}
(\Delta-1)\pi^2/4,&\text{if $\Delta$ is odd},\\
(\Delta-2)\pi^2/2,&\text{if $\Delta$ is even}.
\end{cases}
\end{equation}
He modified the conjecture to be
\[
\lim_{n\to\infty}n^2\bigl(\Delta-\lambda_1(n,\Delta)\bigr)=
\begin{cases}
(\Delta-1)\pi^2/4,&\text{if $\Delta$ is odd},\\
(\Delta-2)\pi^2/2,&\text{if $\Delta$ is even},
\end{cases}
\]
and proved the cases $\Delta\in\{3,4\}$.

In this paper, we settle the conjecture for each odd maximum degree  $\Delta\ge 5$.

\begin{theorem}\label{th1}
For each fixed odd integer $\Delta\ge 5$,
\[
\lim_{n\to\infty}\frac{n^2\bigl(\Delta-\lambda_1(n,\Delta)\bigr)}{\Delta-1}=\frac{\pi^2}{4}.
\]
\end{theorem}

Liu's construction \cite{LLL} gives the required $\limsup$ bound in Theorem \ref{th1}; the new contribution is the
matching $\liminf$ bound for each fixed odd integer $\Delta\ge5$.

\section{Preliminaries}\label{sec:preliminaries}

Throughout this paper, $|V(G)|=n$, and vectors in $\mathbb{R}^n$ are indexed by $V(G)$.
Liu's graph constructions give the following upper bound for the spectral gap.

\begin{lem}[{\cite{LLL}}]\label{lem:liu-upper-bound}
For each fixed odd integer $\Delta\ge3$, there are graphs $H_n\in\mathcal{C}(n,\Delta)$ for all sufficiently large $n$ such that
\begin{equation}\label{eq:liu-gap-upper-bound}
0<\Delta-\lambda_1(n,\Delta)\le\Delta-\rho(H_n)=(1+o(1))\frac{(\Delta-1)\pi^2}{4n^2}
\end{equation}
as $n\to\infty$.
\end{lem}

To prove Theorem \ref{th1}, it therefore remains to establish the asymptotic lower bound with the same leading constant.
We first introduce the graph energy and effective resistance.
Write $L=L(G)$ for the \emph{Laplacian matrix} of $G$, which is defined by
\[
(Ly)_v=\sum_{u\in N(v)}(y_v-y_u),\qquad y\in\mathbb{R}^n,\ v\in V.
\]

For any real vector $y$ indexed by $V$, we define its \emph{Dirichlet energy} on $G$ by
\[
\mathcal E_G(y):=\sum_{uv\in E}(y_u-y_v)^2=y^{\mathsf T}Ly.
\]
We use the same energy definition for multigraphs, counting parallel edges separately.
Degrees in a multigraph are counted with multiplicity.

With one reference direction on each edge of a multigraph $H$, write $e=ab$ when that direction is from $a$ to $b$.
Let $\bar e$ denote the same edge with the reverse direction.
A real edge function $j$ is signed relative to the reference directions and extends to reversed edges by antisymmetry:
\[
j(\bar e)=-j(e),\qquad e\in E(H).
\]
Recall that divergence is net outflow:
\[
(\operatorname{div}j)(u):=\sum_{e=ua\in E(H)}j(e)-\sum_{e=au\in E(H)}j(e),\qquad u\in V(H).
\]
Changing a reference direction and the corresponding sign of $j(e)$ leaves the flow and its divergence unchanged.

\begin{samepage}
\begin{defi}\label{def:unit-flow}
Let $H$ be a connected multigraph, and let $v,w\in V(H)$ be distinct.
A \emph{unit $w$--$v$ flow} is a function $j:E(H)\to\mathbb R$ such that for each $u\in V(H)$,
\[
(\operatorname{div}j)(u)=\begin{cases}
1, & \text{if }u=w;\\
-1, &\text{if } u=v;\\
0, &\text{if } u\notin\{v,w\}.
\end{cases}
\]
\end{defi}
\end{samepage}

View $w$ and $v$ as the positive and negative terminals, respectively.
A unit $w$--$v$ flow represents one unit of current entering at $w$ and leaving at $v$.
We define the \emph{energy of a flow} $j$ to be $\sum_{e\in E(H)}j(e)^2$, where each edge is counted once;
physically, it is the Joule dissipation per unit time. This energy is independent of the reference edge directions.

A real function $y:V(H)\to\mathbb R$, with $y_u=y(u)$, represents a \emph{vertex potential}.
The \emph{current induced by $y$} is
\[
j(e)=y_a-y_b,\qquad e=ab\in E(H).
\]
Consequently,
\[
\sum_{e\in E(H)}j(e)^2=\sum_{e=ab\in E(H)}(y_a-y_b)^2=\mathcal E_H(y).
\]
Thus, the Dirichlet energy of $y$ equals the energy of its induced current.

A flow $j$ is induced by a vertex potential if and only if its signed sum along every closed walk is zero.
We consider the least possible energy among all unit flows with the prescribed source $w$ and sink $v$.

\begin{samepage}
\begin{defi}\label{def:effective-resistance}
Let $H$ be a connected multigraph.
For distinct $v,w\in V(H)$, their \emph{effective resistance} $R(v,w)$ is defined by
\[
R(v,w):=\min\left\{\sum_{e\in E(H)}j(e)^2:j\text{ is a unit }w\text{--}v\text{ flow on }H\right\}.
\]
\end{defi}
\end{samepage}
\pagebreak[0]

\begin{samepage}
The unit flows form a nonempty closed affine set in $\mathbb R^{E(H)}$,
so the squared Euclidean norm attains a minimum on this set.
If $j_1$ and $j_2$ both attain the minimum, their average is also a unit flow, and
\[
\sum_{e\in E(H)}\left(\frac{j_1(e)+j_2(e)}2\right)^2=R(v,w)-\frac14\sum_{e\in E(H)}\bigl(j_1(e)-j_2(e)\bigr)^2.
\]
Minimality forces $j_1=j_2$, so the minimizing flow is unique.
\end{samepage}

\begin{samepage}
Let $j_*$ be the unique unit $w$--$v$ flow attaining the minimum in Definition \ref{def:effective-resistance}.
Thomson's principle \cite[Section 1.3.5]{DS} states that $j_*$ is the steady unit current induced by a vertex potential $y$:
\[
j_*(e)=y_a-y_b\qquad(e=ab\in E(H)).
\]
Consequently,
\[
R(v,w)=\min_j\sum_{e\in E(H)}j(e)^2=\sum_{e\in E(H)}j_*(e)^2=\mathcal E_H(y),
\]
where the minimum is over all unit $w$--$v$ flows $j$ on $H$.
\end{samepage}

We also express resistance in terms of potentials with a prescribed difference between the two terminals.

\begin{lem}\label{lem:dirichlet-resistance}
Let $H$ be a connected multigraph, and let $v,w\in V(H)$ be distinct.
For each real function $y:V(H)\to\mathbb R$,
\begin{equation}\label{eq:resistance-energy-comparison}
(y_w-y_v)^2\le R(v,w)\mathcal E_H(y).
\end{equation}
If a vertex potential $q$ induces a unit $w$--$v$ flow of minimum energy, then
\[
q_w-q_v=\mathcal E_H(q)=R(v,w).
\]
Moreover,
\begin{equation}\label{eq:dirichlet-resistance}
\frac1{R(v,w)}=\min\{\mathcal E_H(y):y:V(H)\to\mathbb R,\ y_v=0,\ y_w=1\}.
\end{equation}
The minimum is unchanged if one also requires $0\le y_u\le1$ for all $u\in V(H)$.
\end{lem}

\begin{proof}

Choose a unit $w$--$v$ flow $j_*$ attaining the minimum in Definition \ref{def:effective-resistance}.
For any potential $y$, the term $j_*(e)(y_a-y_b)$ on an edge $e=ab$ contributes $j_*(e)y_a$ at $a$
and $-j_*(e)y_b$ at $b$. Summing the contributions at each vertex $u$ gives the coefficient
$(\operatorname{div}j_*)(u)$ of $y_u$.
By Definition \ref{def:unit-flow}, it follows that
\[
y_w-y_v=\sum_{u\in V(H)}y_u(\operatorname{div}j_*)(u)=\sum_{e=ab\in E(H)}j_*(e)(y_a-y_b).
\]
By Cauchy--Schwarz,
\[
(y_w-y_v)^2\le\left(\sum_{e\in E(H)}j_*(e)^2\right)\mathcal E_H(y)=R(v,w)\mathcal E_H(y).
\]
This proves \eqref{eq:resistance-energy-comparison} and gives the lower bound in \eqref{eq:dirichlet-resistance}.
By Thomson's principle, $j_*$ is induced by a potential $q$.
Applying the preceding summation identity to $q$ gives
\[
q_w-q_v=\sum_{e=ab\in E(H)}j_*(e)(q_a-q_b)=\sum_{e\in E(H)}j_*(e)^2=R(v,w).
\]
Also $\mathcal E_H(q)=R(v,w)$, so $y_u=(q_u-q_v)/R(v,w)$ for $u\in V(H)$
has $y_v=0$, $y_w=1$, and $\mathcal E_H(y)=1/R(v,w)$.
Finally, replacing $y_u$ by $\min\{1,\max\{0,y_u\}\}$ for each $u\in V(H)$
preserves the terminal values and cannot increase any squared difference along an edge.
Hence the restriction to $[0,1]$ does not change the minimum.
\end{proof}

By Lemma \ref{lem:dirichlet-resistance}, scaling the minimizing unit flow and its potential by $I>0$
gives total current $I$ and terminal voltage difference $U=IR(v,w)$.
Hence $R(v,w)=U/I$, and the dissipated power is $I^2R(v,w)=UI$, as in the physical interpretation of resistance.

For a graph consisting of a single edge $e=uv$, the unit $v$--$u$ flow has $|j(e)|=1$, so $R(u,v)=1$.
Thus, under our energy convention, an individual edge represents a unit resistor;
the effective resistance between its endpoints in a larger network can be smaller.

The following lemma shows that this effective resistance satisfies the usual series and parallel laws.

\begin{lem}[{\cite[Section 1.3.4]{DS}}]\label{lem:series-parallel}
Let $m\ge1$, and let $H_1,\ldots,H_m$ be pairwise vertex-disjoint connected multigraphs,
with distinct terminals $u_i,v_i\in V(H_i)$ for $1\le i\le m$.
Write $R_i=R(u_i,v_i)$, computed in $H_i$.
Both constructions below retain every edge and identify only the specified vertices.
\begin{enumerate}
\item Identifying $v_i$ with $u_{i+1}$ for $1\le i<m$ gives the \textit{series connection}, in which
\[
R(u_1,v_m)=\sum_{i=1}^{m}R_i.
\]
\par
\item Identifying $u_1,\ldots,u_m$ with a single vertex $u$ and $v_1,\ldots,v_m$ with a single vertex $v$
gives the \textit{parallel connection}, in which
\[
R(u,v)=\left(\sum_{i=1}^{m}\frac1{R_i}\right)^{-1}.
\]
\par
\end{enumerate}
\end{lem}

We will also use the following basic comparison properties of effective resistance.
When comparing different networks, a subscript on $R$ specifies the network in which the resistance is computed.

\begin{lem}\label{lem:resistance-operations}
Let $H$ be a connected multigraph, and let $v,w\in V(H)$ be distinct.
\begin{enumerate}
\item\label{item:resistance-subgraph} If $K$ is a connected subgraph of $H$ containing $v,w$, then
\[
R_H(v,w)\le R_K(v,w).
\]
\item\label{item:resistance-attachment} Let $F$ be a connected multigraph with $V(F)\cap V(H)=\varnothing$,
and let $a\in V(H)$ and $b\in V(F)$.
If $K$ is obtained from the disjoint union of $H$ and $F$ by adding the edge $ab$, then
\[
R_K(v,w)=R_H(v,w).
\]
\item For each $u\in V(H)\setminus\{v,w\}$,
\[
R_H(v,w)\le R_H(v,u)+R_H(u,w).
\]
\end{enumerate}
\end{lem}

\begin{proof}
For (\ref{item:resistance-subgraph}), every unit $w$--$v$ flow on $K$ extends by zero to a unit flow on $H$ with the same energy.
Taking minima gives the inequality.

For (\ref{item:resistance-attachment}), summing the divergences of any unit $w$--$v$ flow on $K$ over $V(F)$ gives zero current on $ab$.
Its restriction to $H$ is therefore a unit flow with no greater energy, so $R_H(v,w)\le R_K(v,w)$.
Conversely, a minimizing flow on $H$ extends by zero over $F$ and $ab$, giving $R_K(v,w)\le R_H(v,w)$.

For (3), let $f$ and $g$ be the electrical potentials of unit flows from $v$ to $u$ and from $u$ to $w$,
respectively. Their sum is the electrical potential of a unit flow from $v$ to $w$. The maximum principle gives
$f_w\ge f_u$ and $g_v\le g_u$, and hence
\[
R_H(v,w)=(f_v-f_w)+(g_v-g_w)\le(f_v-f_u)+(g_u-g_w)=R_H(v,u)+R_H(u,w).
\]
\end{proof}

\begin{samepage}
\begin{defi}\label{def:associated-network}
Let $G$ be a connected nonregular graph with maximum degree $\Delta$.
Retain all vertices and edges of $G$, adjoin a new vertex $v_0$,
and join each $v\in V(G)$ to $v_0$ by $\Delta-d_G(v)$ parallel edges.
We call the resulting multigraph $H$ the \emph{electrical network associated with $G$}.
Every real function $y:V(G)\to\mathbb R$ is extended to $V(H)$ by retaining its values on $V(G)$ and setting $y_{v_0}=0$.
Under this extension convention, $v_0$ is called the \emph{zero-potential vertex} of $H$.
\end{defi}
\end{samepage}

Since $G$ is connected and nonregular, $H$ is connected, and each vertex in $V(G)$ has degree $\Delta$ in $H$.

 By Lemma \ref{lem:dirichlet-resistance} and the Perron energy identity used by Liu \cite{LLL}, we have the following.

\begin{lem}\label{lem:laplacian-energy}
Let $G$ be a connected nonregular graph with maximum degree $\Delta$,
and construct $H$ from $G$ as in Definition \ref{def:associated-network}.
Let $x$ be a positive Perron vector of $G$ normalized by $\sum_{v\in V(G)}x_v^2=n$,
and extend $x$ to $H$ by setting $x_{v_0}=0$. Then
\begin{equation}\label{eq:grounded-perron-energy}
\mathcal E_H(x)=n\bigl(\Delta-\rho(G)\bigr).
\end{equation}
If $w\in V(G)$ satisfies $x_w=\max_{v\in V(G)}x_v$, then
\begin{equation}\label{eq:resistance-gap-lower-bound}
\Delta-\rho(G)\ge\frac{x_w^2}{nR(v_0,w)}\ge\frac{1}{nR(v_0,w)}.
\end{equation}
\end{lem}

\begin{proof}
By the Perron equation on $G$, we have $\mathcal E_G(x)=\sum_{v\in V(G)}(d_G(v)-\rho(G))x_v^2$. Hence
\[
(\Delta-\rho(G))\sum_{v\in V(G)}x_v^2=\mathcal E_G(x)+\sum_{v\in V(G)}(\Delta-d_G(v))x_v^2.
\]
For each $v\in V(G)$, the added edges from $v$ to $v_0$ contribute
$(\Delta-d_G(v))(x_v-x_{v_0})^2=(\Delta-d_G(v))x_v^2$ to the energy. Therefore, by the normalization of $x$,
\[
\mathcal E_H(x)=\mathcal E_G(x)+\sum_{v\in V(G)}(\Delta-d_G(v))x_v^2=n\bigl(\Delta-\rho(G)\bigr).
\]
By \eqref{eq:resistance-energy-comparison} and \eqref{eq:grounded-perron-energy},
\[
x_w^2\le R(v_0,w)\mathcal E_H(x)=nR(v_0,w)\bigl(\Delta-\rho(G)\bigr).
\]
Since $x_w^2\ge n^{-1}\sum_{v\in V(G)}x_v^2=1$, the result follows.
\end{proof}

\section{Resistance estimates}

Throughout this section, $\Delta\ge5$ is an odd integer.
Lemma \ref{lem:laplacian-energy} shows how upper bounds on effective resistance give lower bounds on the spectral gap.
Our aim is therefore to bound effective resistance in terms of the number of vertices under the degree constraints.
We begin with the following estimate for simple graphs, in which the coefficient of the graph order is $1/(\Delta-1)$.

\begin{samepage}
\begin{theo}\label{thm:simple-resistance}
Let $G$ be a connected simple graph of order $n$ with distinct vertices $s,t$.
Suppose that $d_G(s),d_G(t)\le\Delta-1$ and $d_G(v)=\Delta$ for all $v\in V(G)\setminus\{s,t\}$. Then
\begin{equation}\label{eq:simple-resistance}
(\Delta-1)R(s,t)\le n+\Delta-1-d_G(s)-d_G(t).
\end{equation}
\end{theo}
\end{samepage}

\begin{proof}
All degrees refer to the original graph $G$.
We bound its terminal resistance using subgraphs and equivalent circuits.

Choose a shortest $s$--$t$ path $P=u_0u_1\cdots u_\ell$, with $u_0=s$ and $u_\ell=t$.
Write $Q:=V(G)\setminus V(P)$, $q:=|Q|$, and $k_v:=|N_G(v)\cap V(P)|$ for $v\in Q$.
The path is induced. Moreover, if $u_i,u_j\in N_G(v)$, then $|i-j|\le2$,
since otherwise the two-edge route through $v$ would shorten $P$. In particular, $k_v\le3$.

Define the following sets and counts for $1\le r\le\ell$:
\[
\begin{alignedat}{2}
A_r&:=\{v\in Q:N_G(v)\cap V(P)=\{u_{r-1},u_r\}\},\quad &a_r&:=|A_r|,\\
B_r&:=\{v\in Q:N_G(v)\cap V(P)=\{u_{r-1},u_{r+1}\}\},\quad &b_r&:=|B_r|,\\
C_r&:=\{v\in Q:N_G(v)\cap V(P)=\{u_{r-1},u_r,u_{r+1}\}\},\quad &c_r&:=|C_r|.
\end{alignedat}
\]
Set $a_r,b_r,c_r$ to zero outside their respective ranges, and put
\[
\sigma:=\sum_{v\in Q}(k_v-1).
\]
Each internal vertex of $P$ has degree  $\Delta$, so
\[
\sum_{v\in Q}k_v=(\Delta-2)(\ell-1)+d_G(s)+d_G(t)-2.
\]
Since $n=\ell+1+q$, this gives the exact identity
\begin{equation}\label{eq:r9-path-count}
(\Delta-1)\ell-\sigma=n+\Delta-1-d_G(s)-d_G(t).
\end{equation}
By \eqref{eq:r9-path-count}, it suffices to prove $R_G(s,t)\le\ell-\sigma/(\Delta-1)$, since this gives
\[
(\Delta-1)R(s,t)\le(\Delta-1)\ell-\sigma=n+\Delta-1-d_G(s)-d_G(t),
\]
which is \eqref{eq:simple-resistance}.

\cla\label{clm:r9-series} For each $W\subseteq\{v\in Q:k_v\ge2\}$ satisfying $|A_r\cap W|\le\Delta-3$ for each $r\in[\ell]$,
there is a subnetwork $G'\subseteq G$ containing $P$ such that
\[
R_G(s,t)\le R_{G'}(s,t)\le\ell-\frac{\sum_{v\in W}(k_v-1)}{\Delta-1}.
\]

\begin{proof}[Proof of Claim~\ref{clm:r9-series}]
We construct a subnetwork $G'\subseteq G$ containing $P$ by retaining selected vertices in $W$
and specified edges joining them to $P$.
By resistance monotonicity, $R_G(s,t)\le R_{G'}(s,t)$.
Interpret $A_r,B_r,C_r$ as empty outside their respective ranges, and put
\[
h_r:=|(A_r\cup C_{r-1}\cup C_r)\cap W|,\qquad m_r:=|(B_r\cup C_r)\cap W|,\qquad r\in\mathbb Z.
\]
For $r\in[\ell]$, $h_r$ counts the common neighbors of $u_{r-1},u_r$ in $W$, so $h_r\le\Delta-2$.
Choose an inclusion-maximal set $\mathcal I\subseteq\{r\in[\ell-1]:m_r>0\}$ with no two consecutive indices, and put
\[
\mathcal J:=[\ell]\setminus\bigl(\mathcal I\cup\{r+1:r\in\mathcal I\}\bigr).
\]
The path is thereby partitioned into two-edge intervals indexed by $\mathcal I$ and one-edge intervals indexed by $\mathcal J$.
Every index $r$ with $m_r>0$ belongs to $\mathcal I$ or is adjacent to an index in $\mathcal I$.

Define the subgraph $G'\subseteq G$ by
\[
\begin{aligned}
V(G')&:=V(P)\cup\{v:r\in\mathcal I,\ v\in(B_r\cup C_r)\cap W\}\\
&\hspace{19mm}\cup\{v:r\in\mathcal J,\ v\in(A_r\cup C_{r-1}\cup C_r)\cap W\},\\
E(G')&:=E(P)\cup\{u_jv:r\in\mathcal I,\ v\in(B_r\cup C_r)\cap W,\ |j-r|=1\}\\
&\hspace{19mm}\cup\{u_jv:r\in\mathcal J,\ v\in(A_r\cup C_{r-1}\cup C_r)\cap W,\ j=r-1\text{ or }j=r\}.
\end{aligned}
\]
For each $r\in\mathcal I$, define the subnetwork $G_r$ as the union of the $m_r+1$ two-edge paths
\[
u_{r-1}wu_{r+1},\qquad w\in\{u_r\}\cup((B_r\cup C_r)\cap W) .
\]
\begin{samepage}
For each $r\in\mathcal J$, define the subnetwork $G_r$ as the union of the edge $u_{r-1}u_r$ and the $h_r$ two-edge paths
\[
u_{r-1}wu_r,\qquad w\in(A_r\cup C_{r-1}\cup C_r)\cap W .
\]
\end{samepage}
The union of the subnetworks $G_r$ over $r\in\mathcal I\cup\mathcal J$ is $G'$.
\begin{samepage}
Vertices in $A_r$ or $B_r$ occur in at most one of these subnetworks.
If $C_r\cap W\ne\varnothing$, then $m_r>0$, so maximality of $\mathcal I$ gives
\[
\mathcal I\cap\{r-1,r,r+1\}\ne\varnothing.
\]
\end{samepage}
Together with the definition of $\mathcal J$, this yields
\[
|\mathcal I\cap\{r\}|+|\mathcal J\cap\{r,r+1\}|\le1.
\]
The left side counts the subnetworks using vertices of $C_r\cap W$.
Thus every vertex of $Q$ belongs to at most one of these subnetworks.
Since the path intervals partition $P$, consecutive subnetworks intersect exactly in their common terminal,
and nonconsecutive subnetworks are vertex-disjoint.
Their terminal resistances therefore add by the series law.

\begin{figure}[H]
\centering
\begin{tikzpicture}[
    font=\small,
    line cap=round,
    line join=round,
    vertex/.style={circle,fill=black,inner sep=1.55pt,outer sep=0pt},
    path/.style={line width=.8pt},
    branch/.style={line width=.65pt},
    block/.style={draw=black!48,fill=black!2,rounded corners=4pt,line width=.55pt}
]
\coordinate (figs) at (0,0);
\coordinate (figurm) at (2.4,0);
\coordinate (figur) at (3.65,0);
\coordinate (figurp) at (4.9,0);
\coordinate (figurpm) at (8.15,0);
\coordinate (figurpzero) at (10.05,0);
\coordinate (figt) at (12.5,0);

\draw[block] (2.15,-.72) rectangle (5.15,2.75);
\draw[block] (7.9,-.72) rectangle (10.3,2.75);
\draw[path] (figs)--(figurm)--(figur)--(figurp)--(figurpm)--(figurpzero)--(figt);

\node[above=7pt] at (.55,0) {$P$};
\node[fill=white,inner sep=2pt] at (1.18,0) {$\cdots$};
\node[fill=white,inner sep=2pt] at (6.52,0) {$\cdots$};
\node[fill=white,inner sep=2pt] at (11.3,0) {$\cdots$};

\coordinate (figwone) at (3.65,2.2);
\coordinate (figwlast) at (3.65,1.05);
\draw[branch] (figurm)--(figwone)--(figurp);
\draw[branch] (figurm)--(figwlast)--(figurp);

\coordinate (figzone) at (9.1,2.2);
\coordinate (figzlast) at (9.1,1.05);
\draw[branch] (figurpm)--(figzone)--(figurpzero);
\draw[branch] (figurpm)--(figzlast)--(figurpzero);

\foreach \x in {figs,figurm,figur,figurp,figurpm,figurpzero,figt,figwone,figwlast,figzone,figzlast}
    \node[vertex] at (\x) {};

\node[below=5pt] at (figs) {$s=u_0$};
\node[below=5pt] at (figurm) {$u_{r-1}$};
\node[below=5pt] at (figur) {$u_r$};
\node[below=5pt] at (figurp) {$u_{r+1}$};
\node[below=5pt] at (figurpm) {$u_{r'-1}$};
\node[below=5pt] at (figurpzero) {$u_{r'}$};
\node[below=5pt] at (figt) {$t=u_\ell$};

\node[right=4pt] at (figwone) {$w_1$};
\node[above=3pt,font=\scriptsize] at (figwlast) {$w_{m_r}$};
\node[font=\scriptsize] at (3.65,1.85) {$\vdots$};
\node[right=4pt] at (figzone) {$z_1$};
\node[above=3pt,font=\scriptsize] at (figzlast) {$z_{h_{r'}}$};
\node[font=\scriptsize] at (9.1,1.85) {$\vdots$};

\node[font=\footnotesize] at (3.65,-1.08) {$G_r\ (r\in\mathcal I)$};
\node[font=\footnotesize] at (9.1,-1.08) {$G_{r'}\ (r'\in\mathcal J)$};
\end{tikzpicture}
\caption{$G'$.}
\label{fig:r9-series-network}
\end{figure}

For $r\in\mathcal J$, we have $h_r\le\Delta-3$.
Since $d_G(s),d_G(t)\le\Delta-1$ and $d_G(v)=\Delta$ for $v\in V(G)\setminus\{s,t\}$,
each vertex of $P$ has at most $\Delta-2$ neighbors in $Q$. If $h_r=\Delta-2$, then
\[
N_G(u_{r-1})\cap Q=N_G(u_r)\cap Q=(A_r\cup C_{r-1}\cup C_r)\cap W.
\]
Vertices in $B_{r-2}\cup C_{r-2}$ or $B_{r+1}\cup C_{r+1}$ are adjacent in $G$ to exactly one of $u_{r-1},u_r$,
so both sets are empty and $m_{r-2}=m_{r+1}=0$.
By the assumption $|A_r\cap W|\le\Delta-3$,
\[
m_{r-1}+m_r\ge h_r-|A_r\cap W|\ge1.
\]
Since $r\in\mathcal J$ and $m_{r-2}=m_{r+1}=0$, we have
\[
\mathcal I\cap\{r-2,r-1,r,r+1\}=\varnothing.
\]
Thus at least one of $r-1,r$ can be added to $\mathcal I$ without introducing consecutive indices, contradicting its maximality.

\begin{samepage}
For $r\in\mathcal J$, we have $R_{G_r}(u_{r-1},u_r)=2/(h_r+2)$ by Lemma \ref{lem:series-parallel},
and $h_r\le\Delta-3$ gives
\begin{equation}\label{eq:r9-single-resistance}
\frac{2}{h_r+2}\le1-\frac{h_r}{\Delta-1}.
\end{equation}
\end{samepage}
\begin{samepage}
For $r\in\mathcal I$, we have $R_{G_r}(u_{r-1},u_{r+1})=2/(m_r+1)$, where $1\le m_r\le\Delta-2$.
Counting the neighbors of $u_r$ in $W$ gives
\[
h_r+h_{r+1}+\sum_{j=r-1}^{r+1}|B_j\cap W|\le\Delta-2+m_r.
\]
Since $(m_r-1)(\Delta-2-m_r)\ge0$, it follows that
\begin{equation}\label{eq:r9-pair-resistance}
\frac{2}{m_r+1}\le2-\frac{\Delta-2+m_r}{\Delta-1}\le2-\frac{h_r+h_{r+1}+\sum_{j=r-1}^{r+1}|B_j\cap W|}{\Delta-1}.
\end{equation}
\end{samepage}
Every $h_r$ is counted exactly once in these interval bounds.
Each positive term $|B_r\cap W|$ is counted at least once, since $m_r>0$ and some index in $\{r-1,r,r+1\}$ belongs to $\mathcal I$.

By Lemmas \ref{lem:series-parallel} and \ref{lem:resistance-operations},
\begin{equation}\label{eq:r9-series-resistance}
\begin{aligned}
R_G(s,t)&\le R_{G'}(s,t)=\sum_{r\in\mathcal J}\frac{2}{h_r+2}+\sum_{r\in\mathcal I}\frac{2}{m_r+1}\\
&\le\ell-\frac{\sum_{r=1}^{\ell}h_r+\sum_{r=1}^{\ell-1}|B_r\cap W|}{\Delta-1}=\ell-\frac{\sum_{v\in W}(k_v-1)}{\Delta-1}.
\end{aligned}
\end{equation}

\end{proof}

Put $\mathcal R:=\{r\in[\ell]:a_r=\Delta-2\}$, and define the complementary subsets $W,S$ of $\{v\in Q:k_v\ge2\}$ by
\[
\begin{aligned}
W&:=\{v\in Q:k_v\ge2,\ v\notin A_r\text{ whenever }a_r=\Delta-2\},\\
S&:=\{v\in Q:k_v\ge2\}\setminus W.
\end{aligned}
\]
Distinct sets $A_r$ prescribe distinct neighbor sets on $P$ and are therefore pairwise disjoint.
Also, $k_v=2$ for each $v\in A_r$. Hence
\[
S=\bigsqcup_{r\in\mathcal R}A_r,\qquad
|S|=\sum_{r\in\mathcal R}|A_r|=(\Delta-2)|\mathcal R|.
\]
In particular, $|S|/(\Delta-2)=|\mathcal R|$ is an integer.
Put $n_0:=|\{v\in Q:k_v=0\}|$. It follows that $\sum_{v\in W}(k_v-1)=\sigma-|S|+n_0$.
\par

\begin{samepage}
Since $a_r\le\Delta-2$, we have $A_r\cap W=\varnothing$ when $a_r=\Delta-2$ and $|A_r\cap W|\le a_r\le\Delta-3$ otherwise.
For this $W$, fix $G'$ together with the counts $h_r,m_r$, index sets $\mathcal I,\mathcal J$,
and subnetworks $G_r$ as constructed in the proof of Claim~\ref{clm:r9-series}.
\par
\end{samepage}

\begin{samepage}
For each $r$ with $a_r=\Delta-2$, the set $A_r$ has $\Delta-2$ vertices and each of $u_{r-1},u_r$ has at most $\Delta-2$ neighbors in $Q$, so
\[
N_G(u_{r-1})\cap Q=N_G(u_r)\cap Q=A_r.
\]
\end{samepage}
This equality implies $a_{r-1}=a_{r+1}=0$, so no two indices $r$ satisfying $a_r=\Delta-2$ are consecutive.
Since $A_r\cap W=\varnothing$, neither $u_{r-1}$ nor $u_r$ has a neighbor in $W$. Therefore
\[
h_{r-1}=h_r=h_{r+1}=0,\qquad m_j=0,\quad r-2\le j\le r+1.
\]
By the definitions of $\mathcal I$ and $\mathcal J$, each $j\in\{r-1,r,r+1\}\cap[\ell]$ belongs to $\mathcal J$.
For each such $j$, the subnetwork $G_j$ consists only of the edge $u_{j-1}u_j$.

For each $r$ satisfying $a_r=\Delta-2$, let $G_r'$ be the union of the single-edge network $G_r$ and the $\Delta-2$ two-edge paths
\[
u_{r-1}vu_r,\qquad v\in A_r.
\]
By Lemma \ref{lem:series-parallel},
\[
R_{G_r'}(u_{r-1},u_r)=\left(1+\frac{\Delta-2}{2}\right)^{-1}=\frac2\Delta.
\]
Set $G_+:=G'\cup\bigcup_{r\in\mathcal R}G_r'$. If $S=\varnothing$, then $\mathcal R=\varnothing$ and $G_+=G'$.

The network $G'$ is the series connection of $G_j$ ($j\in\mathcal I\cup\mathcal J$), in increasing index order along $P$.
Each $G_j'$ has the same terminals as $G_j$, and its added vertices are precisely $A_j$.
These sets are pairwise disjoint and lie outside $V(G')$; hence $G_+$ is the series connection, in the same order, of
\[
G_j\quad(j\in\mathcal I\cup\mathcal J,\ a_j\le\Delta-3),\qquad G_j',\quad j\in\mathcal J,\ a_j=\Delta-2.
\]
For each $r$ satisfying $a_r=\Delta-2$, the neighboring single-edge networks $G_{r-1}$ and $G_{r+1}$, when present, belong to the first family.
Thus their edges $u_{r-2}u_{r-1}$ and $u_ru_{r+1}$ are bridges of $G_+$,
and the effective resistance in $G_+$ between the endpoints of each is $1$.

\begin{figure}[H]
\centering
\begin{tikzpicture}[
    font=\small,
    line cap=round,
    line join=round,
    vertex/.style={circle,fill=black,inner sep=1.5pt,outer sep=0pt},
    path/.style={line width=.8pt},
    retained/.style={line width=1.25pt},
    branch/.style={line width=.65pt},
    added/.style={draw=black!68,fill=black!5,rounded corners=5pt,line width=.75pt}
]
\coordinate (figpluss) at (0,0);
\coordinate (figplusjm) at (4.35,0);
\coordinate (figplusj) at (6.15,0);
\coordinate (figplust) at (10.5,0);

\draw[added] (4.1,-.72) rectangle (6.4,2.95);
\draw[path] (figpluss)--(figplusjm)--(figplusj)--(figplust);
\draw[retained] (figplusjm)--(figplusj);

\node[above=7pt] at (.55,0) {$P$};
\node[fill=white,inner sep=2pt] at (2.25,0) {$\cdots$};
\node[fill=white,inner sep=2pt] at (8.25,0) {$\cdots$};

\coordinate (figplusvone) at (5.25,2.4);
\coordinate (figplusvlast) at (5.25,1.05);
\draw[branch] (figplusjm)--(figplusvone)--(figplusj);
\draw[branch] (figplusjm)--(figplusvlast)--(figplusj);

\foreach \x in {figpluss,figplusjm,figplusj,figplust,figplusvone,figplusvlast}
    \node[vertex] at (\x) {};

\node[below=5pt] at (figpluss) {$s=u_0$};
\node[below=5pt] at (figplusjm) {$u_{j-1}$};
\node[below=5pt] at (figplusj) {$u_j$};
\node[below=5pt] at (figplust) {$t=u_\ell$};

\node[right=3pt] at (figplusvone) {$v_1$};
\node[above=3pt,font=\scriptsize] at (figplusvlast) {$v_{\Delta-2}$};
\node[font=\scriptsize] at (5.25,1.98) {$\vdots$};

\node[above=4pt] at (5.25,0) {$G_j$};
\node[font=\footnotesize] at (5.25,-1.08) {$G_j'\ (a_j=\Delta-2)$};
\end{tikzpicture}
\caption{$G_+$.}
\label{fig:r9-equality-network}
\end{figure}

\begin{samepage}
Using \eqref{eq:r9-series-resistance} and the identity for $\sum_{v\in W}(k_v-1)$, we obtain
\begin{equation}\label{eq:r9-equality-base}
\begin{aligned}
R_{G_+}(s,t)&=R_{G'}(s,t)-\frac{|S|}{\Delta}\\
&\le\ell-\frac{\sigma}{\Delta-1}+\frac{|S|}{\Delta(\Delta-1)}-\frac{n_0}{\Delta-1}.
\end{aligned}
\end{equation}
\end{samepage}

If $|\mathcal R|\le n_0$, then
\[
\frac{|S|}{\Delta(\Delta-1)}
=\frac{(\Delta-2)|\mathcal R|}{\Delta(\Delta-1)}
\le\frac{n_0}{\Delta-1},
\]
so \eqref{eq:r9-equality-base} gives the required bound. We may therefore assume that $|\mathcal R|>n_0$.

\begin{samepage}
\cla\label{clm:r9-bypass-saving} If $|\mathcal R|>n_0$, then
\[
R_G(s,t)\le R_{G_+}(s,t)-\frac{\Delta-2}{\Delta(\Delta-1)}
\left(|\mathcal R|-n_0\right).
\]
\end{samepage}

\begin{proof}[Proof of Claim~\ref{clm:r9-bypass-saving}]
We first compute, for later use, the resistance in $G_r'$ from $v\in A_r$ to either terminal. Temporarily denote by
$G_r'-v$ the network obtained from $G_r'$ by deleting $v$ and its incident edges. Between $u_{r-1}$ and $u_r$,
this auxiliary network is the parallel connection of the edge $u_{r-1}u_r$ and $\Delta-3$ two-edge paths. Hence
\[
R_{G_r'-v}(u_{r-1},u_r)=\left(1+\frac{\Delta-3}{2}\right)^{-1}=\frac{2}{\Delta-1}.
\]
In $G_r'$, the edge $vu_{r-1}$ is in parallel with the series connection of $vu_r$ and $G_r'-v$. By
Lemma \ref{lem:series-parallel},
\[
\begin{aligned}
R_{G_r'}(v,u_{r-1})
&=\left(1+\frac1{1+R_{G_r'-v}(u_r,u_{r-1})}\right)^{-1}\\
&=\left(1+\frac1{1+2/(\Delta-1)}\right)^{-1}=\frac{\Delta+1}{2\Delta}.
\end{aligned}
\]
By symmetry, $R_{G_r'}(v,u_r)=(\Delta+1)/(2\Delta)$ as well.

Whenever a construction below admits more than one admissible choice, fix one such choice and retain it throughout
the proof. Accordingly, $x_r$ and $e_r$, whenever introduced, denote one fixed vertex and one fixed edge,
respectively; no uniqueness among the admissible candidates is asserted.

For each fixed $r\in\mathcal R$, exactly one of the following two cases holds.

\medskip
\noindent\textit{Case 1.} There exists
$
 r'\in\mathcal R\setminus\{r\}$ such that $E(G)$ contains an edge $vw$ with $v\in A_r,\ w\in A_{r'}$.
Fix one such triple $(r',v,w)$. By symmetry, it suffices to consider $r<r'$. Since the indices in $\mathcal R$ are
nonconsecutive, $r'-r\ge2$. The route $u_{r-1}vwu_{r'}$ has length $3$, so the shortestness of $P$ gives
$r'-r+1\le3$, and hence $r'=r+2$. Set $e_r:=u_ru_{r+1}$. The subnetwork
$G_r'\cup G_{r+2}'\cup\{vw\}$ of $G-e_r$ connects $u_r$ and $u_{r+1}$.
By Lemma \ref{lem:series-parallel}, its resistance is
\[
R_{G_r'}(u_r,v)+1+R_{G_{r+2}'}(w,u_{r+1})=2\left(\frac{\Delta+1}{2\Delta}\right)+1=2+\frac1\Delta.
\]

\begin{figure}[H]
\centering
\begin{tikzpicture}[font=\small,line cap=round,line join=round]
\coordinate (um) at (0,0);
\coordinate (u) at (3.2,0);
\coordinate (up) at (6.2,0);
\coordinate (upp) at (9.4,0);
\coordinate (v) at (1.6,1.2);
\coordinate (zone) at (1.6,2.15);
\coordinate (zlast) at (1.6,3.15);
\coordinate (vp) at (7.8,1.2);
\coordinate (yone) at (7.8,2.15);
\coordinate (ylast) at (7.8,3.15);

\node[rblockbox,fit=(um) (u) (zlast),inner xsep=10pt,inner ysep=10pt] (leftblock) {};
\node[rblockbox,fit=(up) (upp) (ylast),inner xsep=10pt,inner ysep=10pt] (rightblock) {};
\node[anchor=north west,font=\footnotesize] at ([xshift=3pt,yshift=-3pt]leftblock.north west) {$G_r'$};
\node[anchor=north east,font=\footnotesize] at ([xshift=-3pt,yshift=-3pt]rightblock.north east) {$G_{r+2}'$};

\draw[rpath] (um)--(u);
\draw[rextra] (u)--(up);
\draw[rpath] (up)--(upp);
\draw[rchosenedge] (um)--(v)--(u);
\draw[rlocaledge] (um)--(zone)--(u);
\draw[rlocaledge] (um)--(zlast)--(u);
\draw[rchosenedge] (up)--(vp)--(upp);
\draw[rlocaledge] (up)--(yone)--(upp);
\draw[rlocaledge] (up)--(ylast)--(upp);
\draw[rchosenedge] (v)--(vp);

\foreach \p in {um,u,up,upp}
    \node[rpathvertex] at (\p) {};
\foreach \p in {zone,zlast,yone,ylast}
    \node[rinnervertex] at (\p) {};
\node[rchosenvertex] at (v) {};
\node[rchosenvertex] at (vp) {};

\node[below=10pt] at (um) {$u_{r-1}$};
\node[below=10pt] at (u) {$u_r$};
\node[below=10pt] at (up) {$u_{r+1}$};
\node[below=10pt] at (upp) {$u_{r+2}$};
\node[above left=2pt,fill=rblockfill,inner sep=1pt] at (v) {$v$};
\node[above right=2pt,fill=rblockfill,inner sep=1pt] at (vp) {$w$};
\node at (1.6,2.67) {$\vdots$};
\node at (7.8,2.67) {$\vdots$};
\node[fill=rblockfill,inner sep=1pt] at (2.12,2.63) {$\Delta-3$};
\node[fill=rblockfill,inner sep=1pt] at (7.28,2.63) {$\Delta-3$};
\node[below=4pt] at ($(u)!0.5!(up)$) {$e_r=u_ru_{r+1}$};
\end{tikzpicture}
\caption{$G_r'\cup G_{r+2}'\cup\{vw\}$.}
\label{fig:r9-direct-bypass}
\end{figure}
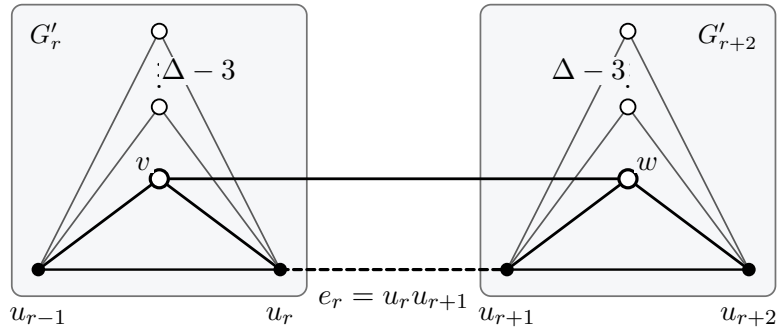

\medskip
\noindent\textit{Case 2.} For each $r'\in\mathcal R\setminus\{r\}$, $vw\notin E(G)$ for all
$ v\in A_r,\ w\in A_{r'}.
$
Fix one $v_r\in A_r$.
Since
\[
|N_G(v_r)\cap(V(P)\cup A_r)|\le2+(\Delta-3)=\Delta-1<d_G(v_r),
\]
$v_r$ has a neighbor outside $V(P)\cup A_r$. By Case 2, such a neighbor lies in $Q\setminus S$, and hence
$N_G(v_r)\cap(Q\setminus S)\ne\varnothing$. Fix one vertex in this set and denote it by $x_r$. Make these choices
once for every index satisfying Case 2 and retain them throughout the proof. Thus $x_r$ is the fixed choice for
$r$, while $x_r=x_{r'}$ is allowed for distinct indices $r$ and $r'$.
The shortestness of $P$ and $N_G(u_{r-1})\cap Q=N_G(u_r)\cap Q=A_r$ give
\[
\{j:u_j\in N_G(x_r)\}\subseteq\{r-3,r-2,r+1,r+2\}.
\]
Recall that $k_{x_r}=|N_G(x_r)\cap V(P)|$.
\par
\smallskip
\noindent\textit{Case 2(a): $k_{x_r}>0$.}
Fix one $u_{j_r}\in N_G(x_r)\cap V(P)$.
The cases $j_r\le r-2$ and $j_r\ge r+1$ are symmetric, so it suffices to describe the latter.
Set $e_r:=u_ru_{r+1}$. Take $G_r'$ with terminals $u_r,v_r$, add the edges $v_rx_r$ and $x_ru_{j_r}$,
and also add $u_{r+1}u_{r+2}$ when $j_r=r+2$. This gives a subnetwork of $G-e_r$ connecting $u_r$ and
$u_{r+1}$ whose resistance is at most
\[
R_{G_r'}(u_r,v_r)+3
=\frac{\Delta+1}{2\Delta}+3
\le4+\frac1\Delta.
\]
When $j_r\le r-2$, use the reflected construction and set $e_r:=u_{r-2}u_{r-1}$; the same bound follows.

\begin{figure}[H]
\centering
\begin{tikzpicture}[font=\small,line cap=round,line join=round]
\coordinate (um) at (0,0);
\coordinate (u) at (3.4,0);
\coordinate (up) at (7.2,0);
\coordinate (upp) at (10.2,0);
\coordinate (v) at (1.7,1.25);
\coordinate (zone) at (1.7,2.2);
\coordinate (zlast) at (1.7,3.2);
\coordinate (x) at (6.15,2.35);

\node[rblockbox,fit=(um) (u) (zlast),inner xsep=10pt,inner ysep=10pt] (block) {};
\node[anchor=north west,font=\footnotesize] at ([xshift=3pt,yshift=-3pt]block.north west) {$G_r'$};

\draw[rpath] (um)--(u);
\draw[rextra] (u)--(up);
\draw[rpath] (up)--(upp);
\draw[rchosenedge] (um)--(v)--(u);
\draw[rlocaledge] (um)--(zone)--(u);
\draw[rlocaledge] (um)--(zlast)--(u);
\draw[rchosenedge] (v)--(x)--(upp);

\foreach \p in {um,u,up,upp}
    \node[rpathvertex] at (\p) {};
\foreach \p in {zone,zlast,x}
    \node[rinnervertex] at (\p) {};
\node[rchosenvertex] at (v) {};

\node[below=10pt] at (um) {$u_{r-1}$};
\node[below=10pt] at (u) {$u_r$};
\node[below=10pt] at (up) {$u_{r+1}$};
\node[below=10pt] at (upp) {$u_{r+2}$};
\node[above left=2pt,fill=rblockfill,inner sep=1pt] at (v) {$v_r$};
\node at (1.7,2.72) {$\vdots$};
\node[fill=rblockfill,inner sep=1pt] at (2.22,2.68) {$\Delta-3$};
\node[above=5pt] at (x) {$x_r$};
\node[below=4pt] at ($(u)!0.5!(up)$) {$e_r=u_ru_{r+1}$};
\end{tikzpicture}
\caption{$G_r'\cup\{v_rx_r,x_ru_{r+2},u_{r+1}u_{r+2}\}$.}
\label{fig:r9-positive-receiver-bypass}
\end{figure}
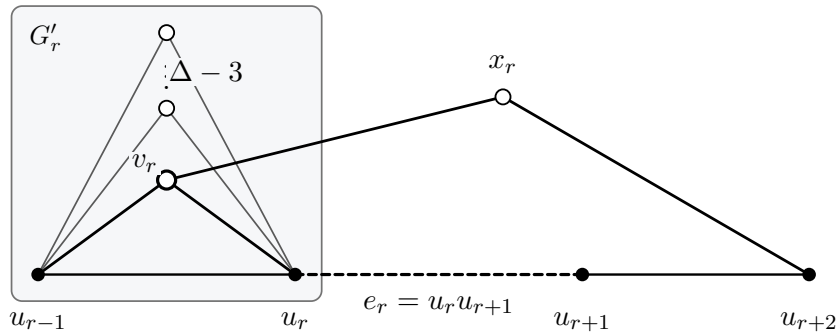

\smallskip
\noindent\textit{Case 2(b): $k_{x_r}=0$.}
Consider a vertex appearing among the fixed choices $x_r$ for the indices $r$ in Case 2(b). If it appears as both
$x_r$ and $x_{r'}$, where $r<r'$, then $r'-r\ge2$ because the indices in $\mathcal R$ are nonconsecutive. The
route $u_{r-1}v_rx_rv_{r'}u_{r'}$ has length $4$, so the shortestness of $P$ gives $r'-r+1\le4$. Hence
\[
2\le r'-r\le3.
\]
No vertex can appear for three indices, since the smallest and largest of three nonconsecutive indices would differ
by at least $4$. Thus each vertex appears at most twice. If it appears once, select no edge. If it appears twice,
let $r<r'$ be the corresponding indices and set $e_r:=u_ru_{r+1}$. Take $G_r'$ with terminals $u_r,v_r$ and
$G_{r'}'$ with terminals $v_{r'},u_{r'-1}$, and add the edges $v_rx_r$, $x_rv_{r'}$, and $u_i u_{i+1}$ for
$r+1\le i\le r'-2$. The resulting subnetwork of $G-e_r$ connects $u_r$ and $u_{r+1}$. By Lemma
\ref{lem:series-parallel}, its resistance is
\[
R_{G_r'}(u_r,v_r)+2+(r'-r-2)+R_{G_{r'}'}(v_{r'},u_{r'-1})
=\frac{\Delta+1}{\Delta}+(r'-r)\le4+\frac1\Delta.
\]

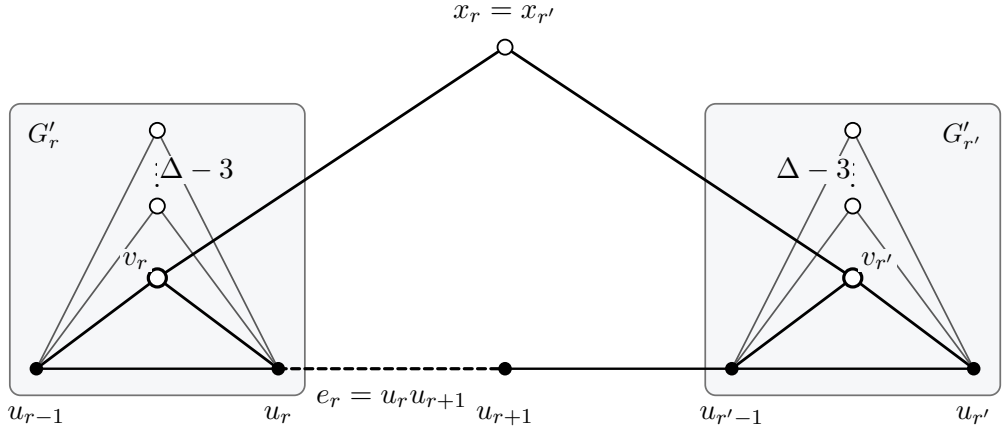
\begin{figure}[H]
\centering
\begin{tikzpicture}[font=\small,line cap=round,line join=round]
\coordinate (um) at (0,0);
\coordinate (u) at (3.2,0);
\coordinate (up) at (6.2,0);
\coordinate (urp) at (9.2,0);
\coordinate (ur) at (12.4,0);
\coordinate (v) at (1.6,1.2);
\coordinate (zone) at (1.6,2.15);
\coordinate (zlast) at (1.6,3.15);
\coordinate (vp) at (10.8,1.2);
\coordinate (yone) at (10.8,2.15);
\coordinate (ylast) at (10.8,3.15);
\coordinate (x) at (6.2,4.25);

\node[rblockbox,fit=(um) (u) (zlast),inner xsep=10pt,inner ysep=10pt] (leftblock) {};
\node[rblockbox,fit=(urp) (ur) (ylast),inner xsep=10pt,inner ysep=10pt] (rightblock) {};
\node[anchor=north west,font=\footnotesize] at ([xshift=3pt,yshift=-3pt]leftblock.north west) {$G_r'$};
\node[anchor=north east,font=\footnotesize] at ([xshift=-3pt,yshift=-3pt]rightblock.north east) {$G_{r'}'$};

\draw[rpath] (um)--(u);
\draw[rextra] (u)--(up);
\draw[rpath] (up)--(urp)--(ur);
\draw[rchosenedge] (um)--(v)--(u);
\draw[rlocaledge] (um)--(zone)--(u);
\draw[rlocaledge] (um)--(zlast)--(u);
\draw[rchosenedge] (urp)--(vp)--(ur);
\draw[rlocaledge] (urp)--(yone)--(ur);
\draw[rlocaledge] (urp)--(ylast)--(ur);
\draw[rchosenedge] (v)--(x)--(vp);

\foreach \p in {um,u,up,urp,ur}
    \node[rpathvertex] at (\p) {};
\foreach \p in {zone,zlast,yone,ylast,x}
    \node[rinnervertex] at (\p) {};
\node[rchosenvertex] at (v) {};
\node[rchosenvertex] at (vp) {};

\node[below=10pt] at (um) {$u_{r-1}$};
\node[below=10pt] at (u) {$u_r$};
\node[below=10pt] at (up) {$u_{r+1}$};
\node[below=10pt] at (urp) {$u_{r'-1}$};
\node[below=10pt] at (ur) {$u_{r'}$};
\node[above left=2pt,fill=rblockfill,inner sep=1pt] at (v) {$v_r$};
\node[above right=2pt,fill=rblockfill,inner sep=1pt] at (vp) {$v_{r'}$};
\node at (1.6,2.67) {$\vdots$};
\node at (10.8,2.67) {$\vdots$};
\node[fill=rblockfill,inner sep=1pt] at (2.12,2.63) {$\Delta-3$};
\node[fill=rblockfill,inner sep=1pt] at (10.28,2.63) {$\Delta-3$};
\node[above=5pt] at (x) {$x_r=x_{r'}$};
\node[below=4pt] at ($(u)!0.5!(up)$) {$e_r=u_ru_{r+1}$};
\end{tikzpicture}
\caption{$G_r'\cup G_{r'}'\cup\{v_rx_r,x_rv_{r'},u_{r+1}u_{r'-1}\}$ for $r'-r=3$.}
\label{fig:r9-zero-receiver-bypass}
\end{figure}

Recall that $\mathcal R=\{r\in[\ell]:a_r=\Delta-2\}$ and $n_0=|\{v\in Q:k_v=0\}|$. After applying the preceding
constructions to all indices in $\mathcal R$, count the resulting selected edges $e_r$ with multiplicity: if the
same path edge is selected by constructions associated with two different indices, it contributes twice. Every
index treated in Case 1 or Case 2(a) produces one selected edge. In Case 2(b), fix a vertex $x$ appearing among
the choices $x_r$. There are one or two indices in Case 2(b) satisfying $x_r=x$, and the construction produces
zero or one selected edge, respectively, one fewer than the number of these indices. Each distinct such vertex
$x$ satisfies $k_x=0$, so there are at most $n_0$ of them. Therefore the preceding constructions produce at least
$|\mathcal R|-n_0$ selected edges $e_r$, counted with multiplicity.

We now prove that every path edge is selected at most twice and that, if it is selected twice, both selections arise
from Case 1. Every selected edge indexed by $q$ is either $u_{q-2}u_{q-1}$ or $u_qu_{q+1}$. Therefore a fixed
path edge $e=u_ru_{r+1}$ can be selected only with indices $q=r$ and $q=r+2$, and hence at most twice.

Suppose that $e$ is selected twice. Then
\[
e_r=e_{r+2}=u_ru_{r+1}.
\]
We first show that the selection indexed by $r$ arises from Case 1. It cannot arise from Case 2(a): in that case
$j_r\in\{r+1,r+2\}$, while
\[
N_G(u_{r+1})\cap Q=N_G(u_{r+2})\cap Q=A_{r+2}
\]
would give $x_r\in A_{r+2}\subseteq S$, contrary to $x_r\in Q\setminus S$. It cannot arise from Case 2(b)
either. Indeed, the other index $r'$ with $x_{r'}=x_r$ would satisfy $r'-r\in\{2,3\}$. If $r'=r+2$, the joint
construction for these two indices makes only the single selection $e_r$, contradicting the second selection
indexed by $r+2$. If $r'=r+3$, then $r+2,r+3\in\mathcal R$ are consecutive, again a contradiction. Thus the
selection indexed by $r$ arises from Case 1. Selecting $u_ru_{r+1}$ in Case 1 forces the other index to be $r+2$;
hence some $v\in A_r$ and $w\in A_{r+2}$ satisfy $vw\in E(G)$. The same edge $vw$ shows that $r+2$ also satisfies
the condition of Case 1. By the case distinction, the selection indexed by $r+2$ therefore arises from Case 1 as
well.

Let $e=ab$ be a distinct selected edge, and let $p\in\{1,2\}$ be its selection multiplicity. In $G_+$, the edge
$e$ is a unit-resistance series block.

If $p=1$, the construction selecting $e$ gives an $a$--$b$ subnetwork of $G-e$ with resistance at most
$4+1/\Delta$. By Lemmas \ref{lem:series-parallel} and \ref{lem:resistance-operations}, placing this subnetwork in
parallel with $e$ gives
\[
1-R_G(a,b)\ge1-\left(1+\frac1{4+1/\Delta}\right)^{-1}
=\frac{\Delta}{5\Delta+1}\ge\frac{\Delta-2}{\Delta(\Delta-1)}.
\]

If $p=2$, both selections arise from Case 1. No independence between the two constructions is needed: the Case 1
calculation gives an $a$--$b$ subnetwork of $G-e$ with resistance
\[
2\left(\frac{\Delta+1}{2\Delta}\right)+1=2+\frac1\Delta.
\]
The same lemmas now give
\[
1-R_G(a,b)\ge1-\left(1+\frac1{2+1/\Delta}\right)^{-1}
=\frac{\Delta}{3\Delta+1}\ge\frac{2(\Delta-2)}{\Delta(\Delta-1)}.
\]
Thus, in both cases,
\[
1-R_G(a,b)\ge\frac{\Delta-2}{\Delta(\Delta-1)}p.
\]

List the series blocks of $G_+$ in their order from $s$ to $t$. The triangle inequality and monotonicity in Lemma
\ref{lem:resistance-operations} bound the contribution of every unselected block by its resistance in $G_+$, while
a selected edge $e=ab$ contributes at most
$1-\frac{\Delta-2}{\Delta(\Delta-1)}p$. Summing over the distinct selected edges counts every selection according
to its multiplicity. Since there are at least $|\mathcal R|-n_0$ such selections,
\[
R_G(s,t)\le R_{G_+}(s,t)-\frac{\Delta-2}{\Delta(\Delta-1)}(|\mathcal R|-n_0).
\]
\end{proof}

Combining Claim~\ref{clm:r9-bypass-saving} with \eqref{eq:r9-equality-base}, we obtain
\begin{equation}\label{eq:r9-completed-resistance}
R_G(s,t)\le\ell-\frac{\sigma}{\Delta-1}-\frac{2n_0}{\Delta(\Delta-1)}\le\ell-\frac{\sigma}{\Delta-1}.
\end{equation}

Equation \eqref{eq:r9-path-count} now yields \eqref{eq:simple-resistance}.
\end{proof}

\begin{samepage}
The resistance in Lemma \ref{lem:laplacian-energy} is computed in the electrical network $H$ associated with $G$.
In this network, the degree of $v_0$ may grow with $n$, and edges incident with $v_0$ may be parallel,
so Theorem \ref{thm:simple-resistance} does not apply directly. We first reduce $d_H(s)$ and $d_H(t)$ by operations
that do not decrease resistance, then delete the remaining edges incident with $s$ or $t$ and restore the resulting
degree losses by pendant completion graphs. The final simple graph satisfies the hypotheses of Theorem
\ref{thm:simple-resistance}.\par
\end{samepage}

\begin{theo}\label{thm:terminal-resistance}
Let $G$ be a connected loopless multigraph of order $n$ with distinct vertices $s,t$.
Suppose that $G-\{s,t\}$ is simple and $d_G(u)=\Delta$ for all $u\in V(G)\setminus\{s,t\}$. Then
\begin{equation}\label{eq:terminal-resistance}
R(s,t)\le\frac{n-2}{\Delta-1}+(\Delta+1)^2.
\end{equation}
\end{theo}

\begin{proof}
If $\operatorname{dist}_G(s,t)=1$, Lemma \ref{lem:resistance-operations} gives $R(s,t)\le1$, and the assertion follows.
We may therefore assume that $\operatorname{dist}_G(s,t)\ge2$. Choose a shortest $s$--$t$ path
$P=(u_0,u_1,\ldots,u_\ell)$, where $u_0=s$ and $u_\ell=t$.
Fix an edge $e_s$ joining $s$ to $u_1$ and an edge $e_t$ joining $u_{\ell-1}$ to $t$.

\medskip
\noindent\textbf{Operation.}
Denote the current graph by $G$. Choose $\tau\in\{s,t\}$ and distinct edges
$f_a,f_b\in E(G)\setminus\{e_s,e_t\}$, where $f_a$ joins $\tau$ to $a$, $f_b$ joins $\tau$ to $b$, and $a,b$ are
distinct and nonadjacent in $G$. Let $e_{ab}$ be a new edge joining $a$ and $b$, and set
\[
G':=G+e_{ab}-f_a-f_b.
\]
The exclusion $f_a,f_b\notin\{e_s,e_t\}$ gives $P\subseteq G'$. Moreover, $G'-\{s,t\}$ is simple, and every
vertex in $V(G')\setminus\{s,t\}$ has the same degree in $G'$ as in $G$.

By Lemma \ref{lem:dirichlet-resistance}, choose $z:V(G)\to[0,1]$ such that
\[
z_s=0,\quad z_t=1,\qquad \mathcal E_G(z)=\min_{\substack{y:V(G)\to[0,1]\\ y_s=0,\ y_t=1}}\mathcal E_G(y)=\frac1{R_G(s,t)}.
\]
\[
\mathcal E_{G'}(z)-\mathcal E_G(z)=(z_a-z_b)^2-(z_a-z_\tau)^2-(z_b-z_\tau)^2
=\begin{cases}
-2z_az_b,&\tau=s,\\
-2(1-z_a)(1-z_b),&\tau=t.
\end{cases}
\]
Since $z_a,z_b\in[0,1]$, both expressions are nonpositive, which leads to
$\mathcal E_{G'}(z)\le\mathcal E_G(z)$.
By Lemma \ref{lem:dirichlet-resistance},
\[
\frac1{R_{G'}(s,t)}\le\mathcal E_{G'}(z)\le\mathcal E_G(z)=\frac1{R_G(s,t)},
\]
where $R_{G'}(s,t)$ is computed in the component of $G'$ containing $P$. Hence $R_G(s,t)\le R_{G'}(s,t)$.

Repeat the operation whenever possible, each time retaining the component containing $P$. Since every operation
decreases $d(s)+d(t)$ by two, the process terminates. Denote the resulting graph by $G_1$.
For $\tau\in\{s,t\}$, let $U_\tau$ be the set of vertices in $V(G_1)\setminus\{s,t\}$ joined to $\tau$ by an edge
other than $e_s,e_t$. At termination,
\[
r:=|U_\tau|,\qquad G_1[U_\tau]\cong K_r.
\]
If $r>0$, each vertex of $U_\tau$ has $r-1$ neighbors
in $U_\tau$ and total degree $\Delta$, so at most $\Delta-r+1$ of its incident edges join it to $\tau$.
Consequently,
\[
d_{G_1}(\tau)-1\le r(\Delta+1-r)\le\frac{(\Delta+1)^2}{4}.
\]
The same bound is immediate if $U_\tau=\varnothing$. Hence
\begin{equation}\label{eq:r10-terminal-count}
m:=d_{G_1}(s)+d_{G_1}(t)-2\le\frac{(\Delta+1)^2}{2}.
\end{equation}

Starting from $G_1$, delete the $m$ edges incident with $s$ or $t$ other than $e_s,e_t$, and retain the component containing $P$.
If an edge $\tau u$ is deleted and $u$ remains in this component, then $d(u)$ decreases by one. To restore each such
loss while giving all new vertices degree $\Delta$, use a connected simple graph with one specified vertex $c$ of
degree $\Delta-1$ and all other vertices of degree $\Delta$. Obtain such a graph from $K_{\Delta+2}$ by choosing
distinct vertices $c,a,b$, deleting $ca,cb$, and deleting a perfect matching on the other $\Delta-1$ vertices, which
exists because $\Delta$ is odd.

For every such edge $\tau u$, attach a disjoint copy of this graph by adding the edge $uc$. This restores the lost
degree at $u$ and raises the degree of $c$ from $\Delta-1$ to $\Delta$. The resulting graph $G_2$ is connected and simple,
\[
d_{G_2}(s)=d_{G_2}(t)=1,\qquad d_{G_2}(v)=\Delta\quad\bigl(v\in V(G_2)\setminus\{s,t\}\bigr).
\]
Deleting the terminal edges cannot decrease the $s$--$t$ resistance, whereas each added copy, together with its
joining edge, is attached to the rest of the graph only at $u$ and hence does not change this resistance. By Lemma
\ref{lem:resistance-operations},
\[
R(s,t)\le R_{G_1}(s,t)\le R_{G_2}(s,t).
\]
Since at most $m$ copies, each of order $\Delta+2$, are added,
\[
|V(G_2)|\le n+(\Delta+2)m.
\]

Applying Theorem \ref{thm:simple-resistance} to $G_2$ and using \eqref{eq:r10-terminal-count}, we obtain
\[
\begin{aligned}
R(s,t)&\le R_{G_2}(s,t)\le\frac{|V(G_2)|+\Delta-3}{\Delta-1}\\
&\le\frac{n-2}{\Delta-1}+1+\frac{(\Delta+2)(\Delta+1)^2}{2(\Delta-1)}\\
&\le\frac{n-2}{\Delta-1}+(\Delta+1)^2.
\end{aligned}
\]
For the last inequality, $(\Delta+1)^2\le2\Delta(\Delta-1)$ gives
$(\Delta+2)(\Delta+1)^2/(2(\Delta-1))\le\Delta(\Delta+2)=(\Delta+1)^2-1$.
This proves \eqref{eq:terminal-resistance}.
\end{proof}

We will use the following consequence for the electrical network associated with $G$.

\begin{corollary}\label{cor:shorted-resistance}
Let $G$ be a connected nonregular simple graph with maximum degree $\Delta$,
and construct $H$ from $G$ as in Definition \ref{def:associated-network}.
Let $A,B\subseteq V(H)$ be disjoint nonempty sets with $v_0\in A$.
Obtain $X$ from $H$ by contracting $A$ and $B$ to distinct vertices $a$ and $b$, respectively, retaining parallel
edges and deleting loops, and put
$M=|V(H)\setminus(A\cup B)|$. Then the effective resistance in $X$ satisfies
\begin{equation}\label{eq:shorted-resistance}
R(a,b)\le\frac{M}{\Delta-1}+(\Delta+1)^2.
\end{equation}
\end{corollary}

\begin{samepage}
\begin{proof}
The contractions preserve connectivity and the degree of every uncontracted vertex. Since $v_0\in A$, every vertex
of $X-\{a,b\}$ comes from $V(G)$ and has degree $\Delta$, while $X-\{a,b\}$ is the subgraph of $G$ induced by the
uncontracted vertices and hence is simple. Moreover, $|V(X)|=M+2$. Thus Theorem
\ref{thm:terminal-resistance}, applied to $X$ with terminals $a,b$, gives
\[
R(a,b)\le\frac{|V(X)|-2}{\Delta-1}+(\Delta+1)^2=\frac{M}{\Delta-1}+(\Delta+1)^2.
\]
\end{proof}
\end{samepage}

\section{Asymptotic spectral gap}\label{sec:asymptotic-gap}

Throughout this section, $\Delta\ge5$ is a fixed odd integer.
We use the resistance estimate from the preceding section to obtain the asymptotic upper bound on the spectral radius.
The following form of the discrete coarea formula will be used below; see \cite[Lemma 3.6]{KM}.

\begin{lem}\label{lem:discrete-coarea}
Let $X$ be a finite loopless multigraph and let $y:V(X)\to\mathbb R$. For $t\in\mathbb R$, put
\[
\Phi_y(t):=\sum_{\substack{uv\in E(X)\\\min\{y_u,y_v\}\le t<\max\{y_u,y_v\}}}|y_u-y_v|,
\]
where parallel edges are counted with multiplicity. Then $\Phi_y$ is a compactly supported step function and
\begin{equation}\label{eq:discrete-coarea}
\int_{-\infty}^{\infty}\Phi_y(t)\,dt=\sum_{uv\in E(X)}(y_u-y_v)^2.
\end{equation}
\end{lem}

\begin{proof}
For each edge $uv$, put $a_{uv}:=\min\{y_u,y_v\}$ and
$b_{uv}:=\max\{y_u,y_v\}$. Since the edge multiset is finite, termwise integration gives
\[
\int_{-\infty}^{\infty}\Phi_y(t)\,dt
=\sum_{uv\in E(X)}|y_u-y_v|(b_{uv}-a_{uv})
=\sum_{uv\in E(X)}(y_u-y_v)^2.\qedhere
\]
\end{proof}

For $G\in\mathcal C(n,\Delta)$, construct $H$ from $G$ as in Definition \ref{def:associated-network}.
Let $x$ be any positive real function on $V(G)$, set $x_{v_0}=0$, and put
$M_x:=\max_{v\in V(G)}x_v$.
For $0\le t<M_x$, put
\begin{equation}\label{eq:level-crossing-sum}
\Phi(t):=\sum_{\substack{uv\in E(H)\\\min\{x_u,x_v\}\le t<\max\{x_u,x_v\}}}|x_u-x_v|.
\end{equation}
Thus $\Phi(t)$ is the sum of $|x_u-x_v|$ over the edges joining $\{u:x_u>t\}$ to its complement,
and $\Phi(t)>0$ on $[0,M_x)$. Put
\begin{equation}\label{eq:level-coordinate}
r(t):=\int_0^t\frac{ds}{\Phi(s)}\qquad(0\le t\le M_x).
\end{equation}
Since $\Phi(t)>0$, the function $r$ is continuous and strictly increasing. Hence
$|r(\alpha)-r(\beta)|$ defines a metric on $[0,M_x]$, under which $r$ is an isometry onto $[0,r(M_x)]$.

Set $\ell:=(\Delta-1)r(M_x)/n$, and define $\psi:[0,\ell]\to[0,M_x]$ explicitly in terms of the inverse of $r$ by
\[
\psi(s):=r^{-1}\!\left(\frac{ns}{\Delta-1}\right)\qquad(0\le s\le\ell).
\]
Equivalently, $\psi((\Delta-1)r(t)/n)=t$ for $0\le t\le M_x$.
For each $v\in V(G)$, put $\eta_v:=\psi^{-1}(x_v)$. For a normalized Perron vector, Lemma
\ref{lem:laplacian-energy} can now be rewritten as follows.

\begin{lem}\label{lem:scaled-perron-energy}
Let $G\in\mathcal C(n,\Delta)$, and use the notation
$H,x,\Phi,r,\ell,\psi,\eta_v$ introduced above.
Suppose that $x$ is a Perron vector of $G$ normalized by $\sum_{v\in V(G)}x_v^2=n$.
Then $\psi$ is continuous, strictly increasing, and piecewise linear, $\psi(0)=0$, and
\begin{equation}\label{eq:scaled-perron-identities}
\frac1n\sum_{v\in V(G)}\psi(\eta_v)^2=1,\qquad \int_0^\ell|\psi'(s)|^2\,ds=\frac{n^2\bigl(\Delta-\rho(G)\bigr)}{\Delta-1}.
\end{equation}
\end{lem}

\begin{proof}
Applying Lemma \ref{lem:discrete-coarea} to $H$ and $x$, and using $x_{v_0}=0$ and
$0<x_v\le M_x$ for every $v\in V(G)$, gives
\[
\int_0^{M_x}\Phi(t)\,dt=\sum_{uv\in E(H)}(x_u-x_v)^2=\mathcal E_H(x).
\]
The function $\Phi$ is positive and piecewise constant on $[0,M_x)$, so $r$ and $\psi$ are continuous, strictly
increasing, and piecewise linear. Recalling the defining relation $\psi((\Delta-1)r(t)/n)=t$ and using
$r'(t)=1/\Phi(t)$, differentiation gives
\[
\psi'\!\left(\frac{\Delta-1}{n}r(t)\right)=\frac{n\Phi(t)}{\Delta-1}\qquad\text{for a.e. }t\in(0,M_x).
\]
Consequently, using Lemma \ref{lem:laplacian-energy},
\begin{align*}
\int_0^\ell |\psi'(s)|^2\,ds&=\int_0^{M_x}\left|\psi'\!\left(\frac{\Delta-1}{n}r(t)\right)\right|^2\frac{\Delta-1}{n\Phi(t)}\,dt\\
&=\frac{n}{\Delta-1}\int_0^{M_x}\Phi(t)\,dt=\frac{n^2(\Delta-\rho(G))}{\Delta-1}.
\end{align*}
Since $\psi(\eta_v)=x_v$ for every $v\in V(G)$, the normalization of $x$ gives
\[
\frac1n\sum_{v\in V(G)}\psi(\eta_v)^2=\frac1n\sum_{v\in V(G)}x_v^2=1.
\]
\end{proof}

The next lemma combines a comparison between two thresholds with the resistance estimate in
Corollary \ref{cor:shorted-resistance} to control the spacing of the vertex coordinates.

\begin{lem}\label{lem:level-contraction}
Let $G\in\mathcal C(n,\Delta)$, and construct $H$ from $G$ as in Definition \ref{def:associated-network}.
Let $x$ be any positive real function on $V(G)$, extended by $x_{v_0}=0$,
and define $M_x,\Phi,r,\ell,\psi,\eta_v$ as above. Let $0\le\alpha<\beta\le M_x$.
Obtain $X$ from $H$ by contracting
\[
A:=\{u\in V(H):x_u\le\alpha\},\qquad B:=\{u\in V(H):x_u\ge\beta\}
\]
to distinct vertices $a,b$, respectively, retaining parallel edges and deleting loops.
With $R(a,b)$ computed in $X$, we have
\begin{equation}\label{eq:level-resistance-comparison}
r(\beta)-r(\alpha)\le R(a,b)\le\frac{|\{v\in V(G):\alpha<x_v<\beta\}|}{\Delta-1}+(\Delta+1)^2.
\end{equation}
In particular, the choices $(\alpha,\beta)=(0,M_x)$ and $(\alpha,\beta)=(\psi(s),\psi(t))$ in
\eqref{eq:level-resistance-comparison} give, respectively,
\begin{align}
&\ell\le1+\frac{(\Delta+1)^2(\Delta-1)}n,\label{eq:coordinate-length-bound}\\
&\frac1n\bigl|\{v\in V(G):s<\eta_v<t\}\bigr|\ge t-s-\frac{(\Delta+1)^2(\Delta-1)}n\qquad(0\le s<t\le\ell).\label{eq:coordinate-count-bound}
\end{align}
\end{lem}

\begin{proof}
Define $g:V(H)\to[0,1]$ by setting $g=0$ on $A$, $g=1$ on $B$, and
\[
g(u):=\frac{r(x_u)-r(\alpha)}{r(\beta)-r(\alpha)}\qquad\bigl(u\in V(H)\setminus(A\cup B)\bigr).
\]
The function $g$ descends to a function $f$ on $X$ satisfying $f(a)=0$ and $f(b)=1$. For an edge $uv$, let
$p\le q$ be the two values obtained by projecting $x_u$ and $x_v$ onto $[\alpha,\beta]$. Then
$|g(u)-g(v)|=(r(q)-r(p))/(r(\beta)-r(\alpha))$, so Cauchy--Schwarz gives
\[
|g(u)-g(v)|^2\le\frac{q-p}{(r(\beta)-r(\alpha))^2}\int_p^q\frac{dt}{\Phi(t)^2}.
\]
Since $q-p\le|x_u-x_v|$, summing over all edges shows that each $t\in(\alpha,\beta)$ is counted with total weight
at most $\Phi(t)$. Consequently,
\[
\mathcal E_X(f)\le\frac1{(r(\beta)-r(\alpha))^2}\int_\alpha^\beta\frac{dt}{\Phi(t)}=\frac1{r(\beta)-r(\alpha)}.
\]
Lemma \ref{lem:dirichlet-resistance} now yields $r(\beta)-r(\alpha)\le R(a,b)$. The second inequality in
\eqref{eq:level-resistance-comparison} is Corollary \ref{cor:shorted-resistance}, because its number $M$ of
uncontracted vertices is exactly $|\{v\in V(G):\alpha<x_v<\beta\}|$.

Taking $\alpha=0$ and $\beta=M_x$, and observing that at most $n$ vertices have values strictly between
$\alpha$ and $\beta$, proves \eqref{eq:coordinate-length-bound}. For \eqref{eq:coordinate-count-bound}, let
$0\le s<t\le\ell$ and take $\alpha=\psi(s)$ and $\beta=\psi(t)$. Since $\psi$ is strictly increasing, the vertices
whose values lie strictly between $\alpha$ and $\beta$ are precisely those with $s<\eta_v<t$. The two bounds in
\eqref{eq:level-resistance-comparison} give
\[
\frac{n(t-s)}{\Delta-1}\le\frac{|\{v\in V(G):s<\eta_v<t\}|}{\Delta-1}+(\Delta+1)^2,
\]
which proves \eqref{eq:coordinate-count-bound}.
\end{proof}

The following one-dimensional Poincar\'e inequality with one fixed endpoint is a standard result in mathematical
analysis. We include a short proof for completeness.

\begin{lem}\label{lem:endpoint-poincare}
Let $L>0$, and let $f:[0,L]\to\mathbb R$ be continuous and piecewise continuously differentiable. If $f(0)=0$, then
\[
\int_0^L|f'(s)|^2\,ds\ge\frac{\pi^2}{4L^2}\int_0^L f(s)^2\,ds.
\]
\end{lem}

\begin{proof}
Set $w(s):=\sin(\pi s/(2L))$. For $s\in(0,L]$, write $f(s)=w(s)g(s)$. Since $f(0)=0$ and $f$ is piecewise
continuously differentiable, $g$ remains bounded as $s\downarrow0$. Moreover, $w''=-\pi^2w/(4L^2)$, and hence
\[
|f'|^2-\frac{\pi^2}{4L^2}f^2=w^2|g'|^2+(ww'g^2)'
\]
on every interval on which $f$ is continuously differentiable. Integrating over these intervals, the boundary
terms at their common endpoints cancel. The remaining boundary terms vanish because $w(0)=0$ and $w'(L)=0$.
Therefore
\[
\int_0^L|f'(s)|^2\,ds-\frac{\pi^2}{4L^2}\int_0^L f(s)^2\,ds
=\int_0^Lw(s)^2|g'(s)|^2\,ds\ge0.
\]
\end{proof}

The following elementary one-dimensional estimate converts the coordinate spacing into the sharp constant
$\pi^2/4$.

\begin{lem}\label{lem:ordered-coordinate-bound}
Let $n\ge1$, $\varepsilon\ge0$, and $0<\ell\le1+\varepsilon$. Suppose that
\[
0<\eta_1\le\eta_2\le\cdots\le\eta_n\le\ell,\qquad \eta_i\le\varepsilon+\frac{i-1}{n}\quad(1\le i\le n).
\]
Let $\psi:[0,\ell]\to[0,\infty)$ be continuous, nondecreasing, and piecewise linear, with $\psi(0)=0$. If
\[
\frac1n\sum_{i=1}^n\psi(\eta_i)^2=1,
\]
then
\begin{equation}\label{eq:ordered-coordinate-energy}
\int_0^\ell|\psi'(s)|^2\,ds\ge\frac{\pi^2}{4(1+\varepsilon)^2}.
\end{equation}
\end{lem}

\begin{proof}
Put $L:=1+\varepsilon$ and extend $\psi$ constantly to $[0,L]$, denoting the extension by $\bar\psi$. Since
$\bar\psi^2$ is nondecreasing and $\eta_i\le\varepsilon+(i-1)/n$, we have
\begin{align*}
1&=\frac1n\sum_{i=1}^n\psi(\eta_i)^2\le\frac1n\sum_{i=1}^n\bar\psi\!\left(\varepsilon+\frac{i-1}{n}\right)^2\\
&\le\sum_{i=1}^n\int_{\varepsilon+(i-1)/n}^{\varepsilon+i/n}\bar\psi(s)^2\,ds=\int_\varepsilon^L\bar\psi(s)^2\,ds\le\int_0^L\bar\psi(s)^2\,ds.
\end{align*}
Since $\bar\psi(0)=0$, Lemma \ref{lem:endpoint-poincare} applied to $\bar\psi$ gives
\[
\int_0^\ell|\psi'(s)|^2\,ds=\int_0^L|\bar\psi'(s)|^2\,ds\ge\frac{\pi^2}{4L^2}\int_0^L\bar\psi(s)^2\,ds\ge\frac{\pi^2}{4L^2},
\]
which proves \eqref{eq:ordered-coordinate-energy}.
\end{proof}

The remaining bound in Theorem \ref{th1} is stated below in a finite form. Its application to extremal graphs
complements Liu's construction bound in Lemma \ref{lem:liu-upper-bound}.

\begin{samepage}
\begin{proposition}\label{prop:universal-gap-bound}
For every odd integer $\Delta\ge5$ and every $G\in\mathcal C(n,\Delta)$, put
\[
\varepsilon_n:=\frac{(\Delta+1)^2(\Delta-1)}n.
\]
Then
\begin{equation}\label{eq:universal-gap-finite}
\frac{n^2\bigl(\Delta-\rho(G)\bigr)}{\Delta-1}\ge\frac{\pi^2}{4(1+\varepsilon_n)^2}.
\end{equation}
Consequently, for each fixed odd integer $\Delta\ge5$ and each sequence
$G_n\in\mathcal C(n,\Delta)$ defined for all sufficiently large $n$,
\begin{equation}\label{eq:universal-gap-liminf}
\liminf_{n\to\infty}\frac{n^2\bigl(\Delta-\rho(G_n)\bigr)}{\Delta-1}\ge\frac{\pi^2}{4}.
\end{equation}
In particular,
\begin{equation}\label{eq:asymptotic-spectral-upper-bound}
\lambda_1(n,\Delta)\le\Delta-\frac{(\Delta-1)\pi^2}{4n^2}+o(n^{-2}) \quad\text{as}\quad n\to\infty.
\end{equation}
\end{proposition}
\end{samepage}

\begin{proof}
Let $x$ be a normalized Perron vector of $G$, and define $\ell,\psi$, and $\eta_v$ as above. Since $x_v>0$
for every $v\in V(G)$, we have $\eta_v>0$. By \eqref{eq:coordinate-length-bound}, $\ell\le1+\varepsilon_n$.
Relabel the vertices so that
\[
0<\eta_1\le\eta_2\le\cdots\le\eta_n\le\ell.
\]
Taking $s=0$ and $t=\eta_i$ in \eqref{eq:coordinate-count-bound} gives
\[
\eta_i-\varepsilon_n\le\frac1n\bigl|\{v\in V(G):0<\eta_v<\eta_i\}\bigr|\le\frac{i-1}{n}.
\]
Thus Lemma \ref{lem:ordered-coordinate-bound}, together with \eqref{eq:scaled-perron-identities}, gives
\eqref{eq:universal-gap-finite}. Since $\varepsilon_n\to0$ for fixed $\Delta$, the finite bound implies
\eqref{eq:universal-gap-liminf}. Applying it to an extremal graph for each $n$ gives
\eqref{eq:asymptotic-spectral-upper-bound}.
\end{proof}

\begin{proof}[Proof of Theorem \ref{th1}]
For each fixed odd $\Delta\ge5$, Proposition \ref{prop:universal-gap-bound} gives the required lower limit,
whereas Lemma \ref{lem:liu-upper-bound} gives the reverse upper limit.
\end{proof}

\medskip
\noindent\textbf{Declaration of AI Use.}
During the preparation of this work, OpenAI's GPT-5.6 was used to generate an initial version of the proof. The
authors subsequently adopted a different proof strategy in Section 2 and Section 3, while  the arguments in Section 4 is a modification of the original proof generated by AI.    The authors reviewed and verified all mathematical arguments in the final manuscript and take
full responsibility for its content.

\end{document}